\documentclass[openany, amssymb, psamsfonts]{amsart}
\usepackage{mathrsfs,comment}
\usepackage[usenames,dvipsnames]{color}
\usepackage[normalem]{ulem}
\usepackage{url}
\usepackage[all,arc,2cell]{xy}
\usepackage{amsmath}
\UseAllTwocells
\usepackage{enumerate}
\usepackage{hyperref}
\hypersetup{%
  bookmarksnumbered=true,%
  bookmarks=true,%
  colorlinks=true,%
  linkcolor=blue,%
  citecolor=blue,%
  filecolor=blue,%
  menucolor=blue,%
  pagecolor=blue,%
  urlcolor=blue,%
  pdfnewwindow=true,%
  pdfstartview=FitBH}

\usepackage{listings}
\usepackage{xcolor}
\definecolor{mygreen}{rgb}{0,0.6,0}
\definecolor{mygray}{rgb}{0.5,0.5,0.5}
\definecolor{mymauve}{rgb}{0.58,0,0.82}
\def\makeautorefname#1#2{\expandafter\def\csname#1autorefname\endcsname{#2}}
\def\equationautorefname~#1\null{(#1)\null}
\makeautorefname{footnote}{footnote}%
\makeautorefname{item}{item}%
\makeautorefname{figure}{Figure}%
\makeautorefname{table}{Table}%
\makeautorefname{part}{Part}%
\makeautorefname{appendix}{Appendix}%
\makeautorefname{chapter}{Chapter}%
\makeautorefname{section}{Section}%
\makeautorefname{subsection}{Section}%
\makeautorefname{subsubsection}{Section}%
\makeautorefname{theorem}{Theorem}%
\makeautorefname{thm}{Theorem}%
\makeautorefname{cor}{Corollary}%
\makeautorefname{lem}{Lemma}%
\makeautorefname{prop}{Proposition}%
\makeautorefname{pro}{Property}
\makeautorefname{conj}{Conjecture}%
\makeautorefname{defn}{Definition}%
\makeautorefname{notn}{Notation}
\makeautorefname{notns}{Notations}
\makeautorefname{rem}{Remark}%
\makeautorefname{quest}{Question}%
\makeautorefname{exmp}{Example}%
\makeautorefname{ax}{Axiom}%
\makeautorefname{claim}{Claim}%
\makeautorefname{ass}{Assumption}%
\makeautorefname{asss}{Assumptions}%
\makeautorefname{con}{Construction}%
\makeautorefname{prob}{Problem}%
\makeautorefname{warn}{Warning}%
\makeautorefname{obs}{Observation}%
\makeautorefname{conv}{Convention}%

\newtheorem{thm}{Theorem}[section]

\newtheorem{lem}{Lemma}[section]
\newtheorem{prob}{Problem}[section]

\theoremstyle{definition}
\newtheorem{defn}{Definition}[section]

\newtheorem{exmp}{Example}[section]

\newtheorem*{obs*}{Observation}

\makeatletter
\let\c@obs=\c@thm
\let\c@cor=\c@thm
\let\c@prop=\c@thm
\let\c@lem=\c@thm
\let\c@prob=\c@thm
\let\c@con=\c@thm
\let\c@conj=\c@thm
\let\c@defn=\c@thm
\let\c@notn=\c@thm
\let\c@notns=\c@thm
\let\c@exmp=\c@thm
\let\c@ax=\c@thm
\let\c@pro=\c@thm
\let\c@ass=\c@thm
\let\c@warn=\c@thm
\let\c@rem=\c@thm
\let\c@sch=\c@thm
\makeatother

\usepackage{graphicx}
\usepackage{float}
\usepackage{subfigure}

\usepackage{makecell}

\title{Construction, Extension and Paths of Near-Homogeneous Tournaments}

\author{Rongxia Tang$^{1,\dag}$}
\thanks{$\dag$ Supported by National Key Research and Development Program of China (2020YFA0712500)}
\author{Zhaojun Chen$^{4}$}
\author{Zan-Bo Zhang$^{1,2,3,*}$}
\thanks{$^*$ Corresponding author. Email address: zanbozhang@gdufe.edu.cn \& eltonzhang2001@gmail.com.
 Supported by Guangdong Basic and Applied Basic Research Foundation (2022A1515010900, 2024A1515012286).}
 
 \address{$^1$School of Mathematics, Sun Yat-sen University, 
135 Xingang Road West, Guangzhou 510275, China}
 \address{$^2$ School of Statistics and Mathematics \\
 Guangdong University of Finance \& Economics \\
 Guangzhou 510320, China}
\address{$^3$ Institute of Artificial Intelligence and Deep Learning \\
 Guangdong University of Finance \& Economics \\
 Guangzhou 510320, China}
\address{$^4$ The Division of Physics, Mathematics and Astronomy, California Institute of Technology,
5313 Mountain Meadow Ln, La Canada Flintridge, CA 91011, USA}

\date{\today}

\begin{document}

\begin{abstract}

      A homogeneous tournament is a tournament with $4t+3$ vertices such that every arc is contained in exactly
      $t+1$ cycles of length $3$. Homogeneous tournaments are the first class of tournaments that are proved
      to be path extendable, which means that every nonhamiltonian path $P$ in such a tournament $T$
      can be extended to a path $P'$ with the same initial and terminal vertex and $V(P')=V(P)\cup \{u\}$
      for a certain vertex $u\in V(T)\backslash V(P)$.
      In order to find more path extendable tournaments we study the generalization of homogeneous tournaments
      called near-homogeneous tournaments,
      in which every arc is contained in $t$ or $t+1$ cycles of length $3$.
      Near-homogeneity has been defined in tournaments with $4t+1$ vertices.
      In this paper, we raise a new method to construct near-homogeneous tournaments
      with $4t+1$ vertices.
      We then show that the definition of near-homogeneous tournament can be extended
      to tournaments with an even number of vertices.
      Finally we verify path extendability of near-homogeneous tournaments, thus expand the class of
      path extendable tournaments.


\end{abstract}

\maketitle

\tableofcontents

\section{Introduction}

We consider digraphs that are strict, i.e. without loops or parallel arcs, in this paper.
A \textbf{tournament} $T$ is a digraph consisting of a set of vertices $V(T)$
and a set of arcs (edges with direction) $A(T)$
such that there is exactly one arc between every two vertices $u$ and $v$.
If there is an arc from $u$ to $v$, we say that $u$ is an \textbf{in-neighbor} of $v$ and
$v$ is an \textbf{out-neighbor} of $u$, and denote it by $u\rightarrow v$.
The number of in-neighbors of $u$ is called the \textbf{in-degree} of $u$, denoted by $d^-(u)$.
The number of out-neighbors of $u$ is called the \textbf{out-degree} of $u$, denoted by $d^+(u)$.
The set of out-neighbors of $u$ is called the \textbf{out-neighborhood} of $u$, and
the set of in-neighbors of $u$ in called the \textbf{in-neighborhood} of $u$.
We call out-degree and in-degree of a vertex its \textbf{semi-degrees}.
A \textbf{regular tournament} $T$ is then a tournament in which $d^+(u) = d^-(u)$ for every vertex $u\in V(T)$.
The \textbf{irregularity} of a tournament is defined as $i(T)=\max\{|d^+(u)-d^-(u)|, u \in V(T)\}$,
and can be viewed as a characterization of how close a tournament is to be regular.
Let $V'$ be a subset of $V(T)$, $\langle V' \rangle$ stands for the sub-tournament induced
by $V'$, i.e. the sub-tournament with vertex set $V'$
and the arc set containing all the arcs among vertices in $V'$.
Let $u$ and $v$ be two distinct vertices in a tournament $T$.
A \textbf{$\boldsymbol{(u,v)}$-path} of length $k$, or a \textbf{$\boldsymbol{(u,v)}$-$\boldsymbol{k}$-path},
is a path from $u$ to $v$ with $k$ arcs.
The \textbf{distance} from $u$ to $v$, denoted by $\boldsymbol {d(u,v)}$,
is the length of the shortest $(u,v)$-path.
A \textbf{$\boldsymbol k$-cycle} is a cycle containing $k$ vertices (or equivalently $k$ arcs).
For the terms and concepts not defined in this paper, the reader is referred to \cite{BG2008}.

Paths and cycles in networks or graphs are of great interests to graph theorists.
In search of path and cycle structures, graphs of high density and high symmetry are extensively investigated.
Tournaments are among the classes of dense digraphs.
Hence it is not surprising that they bear rich path and cycle structures.
All tournaments are traceable,
and strong tournaments are hamiltonian, pancyclic, vertex pancyclic and
even cycle extendable if a well-defined exceptional class is excluded.
However, properties involving paths of specific lengths between any vertex pair such as
hamiltonian-connectedness and panconnectedness do not generally hold in strong tournaments,
and symmetry starts to play a role in sufficient conditions for such properties.
Results of Alspach and Alspach et al. (\cite{Alspach1967, AlspachReid1974}) together
showed that for any arc $uv$ in a regular tournament,
there exist a $(v,u)$-path of every length from $2$ to $n-1$,
as well as a $(u,v)$-path of every length from $3$ to $n-1$.
Jakobsen (\cite{Jakobsen1972}) proved that every edge in an almost regular tournament with
$n\ge 8$ vertices is contained in cycles of all lengths $k$, $4 \le k \le n$.
Thomassen (\cite{Thomassen1980}) unified the above results by proving that if the irregularity $i(T)\le \frac{n-9}{5}$, then
a tournament $T$ is strongly panconnected. When it comes to path extendability,
Zhang et al. (\cite{ZhangZhang2017}) proved that in a regular tournament $T$, every path of length from $2$ to $n-1$ is extendable,
unless $T$ belongs to some exceptional graphs that can be characterized. Further in \cite{Zhang2017}, Zhang
proved that homogeneous tournaments are path extendable.

Homogeneous tournaments are the first class of tournaments that are found to be path extendable.
They must be of order $n=4t+3$ for certain integer $t$.
In order to find path extendable tournaments of other orders, we explore
generalizations of homogeneous tournaments that are called near-homogeneous tournaments in this paper.

Symmetry can be viewed as the similarity between parts in an object,
which is reflected in a regular tournament by the fact
that every vertex having the same number of in-neighbors and out-neighbors.
To figure out subsets of regular tournaments that are even more symmetric,
we can require every arc in a tournament $T$ to be contained in a constant number of $3$-cycles,
which gives us the definition of \textbf{homogeneous tournaments},
which is also called \textbf{doubly regular tournaments}.
It is commonly known that a homogeneous tournament $T$ must have $n=4t+3$ vertices, and each arc
$uv$ is contained in $t+1$ $3$-cycles.
We name the number of $3$-cycles containing an arc $uv$ as the \textbf{$\boldsymbol 3$-cyclic-index} of $uv$.

In \cite{Reid and Brown},
Reid and Brown
showed by constructive proofs
that the existence of a homogeneous tournament with $n$ vertices is equivalent to the existence of
a skew Hadamard matrix of order $n+1$, which has a wide application in combinatorial designs
(\cite{ChristosStylianou2008,FMX2015,KX2018,LMX2021}).
Hence the various constructions of skew Hadamard matrices (surveyed in \cite{ChristosStylianou2008})
also gives constructions of homogeneous tournaments.
On the other hand, homogeneous tournament is a special subset of directed strongly regular graphs (DSRG for short),
another topic of rich problems and results, a survey of which can be found in \cite{FengZeng2016}.

A digraph is called \textbf{path extendable} if for any non-hamiltonian $(u,v)$-$k$-path $P$,
there exists a $(u,v)$-$(k+1)$-path $P'$ such that $V(P)\subseteq V(P')$.
Path extendability implies existence of $(u,v)$-path of every length from $d(u,v)$ to $n-1$,
where $n$ denotes the number of vertices in a tournament $T$,
so it is a significant strengthen of \textbf{hamiltonian-connectedness}, which means the
existence of hamiltonian paths of both directions between every vertex pair.
As mentioned above, path extendability of homogeneous tournaments has been verified.
We hope to find path extendable tournaments with $n\not\equiv 3 \pmod 4$ vertices by asking for a symmetry similar to homogeneity in tournaments,
which is the main theme of the current paper.

In \cite{Tabib1980},
Tabib generalized the concept of homogeneity to tournaments with $n=4t+1$ vertices and defined
\textbf{near-homogeneous tournaments} as follows.
\begin{defn}
A tournament $T$ of order $n=4t+1$, $t\ge 1$ is \textbf{near-homogeneous} if any arc of $T$ is contained in $t$ or $t+1$ $3$-cycles.
\end{defn}
Based on a result in \cite{Astie1975},
Asti\'{e} and Dugat in \cite{AstieDugat1992} figured out an exhaustive algorithm to
find a particular kind of near-homogeneous tournaments, i.e. the vertex-transitive tournaments
with two arc orbits; through assistance of computer programming,
they found all six such near-homogeneous tournaments with order less than $1000$,
i.e. with $n\in \{29, 53, 173, 229, 293, 733\}$ vertices.
Later in \cite{Moukouelle1998}, Moukouelle put forward a recursive method
to construct a near-homogeneous tournament of order $16t+13$,
given any homogeneous tournament of order $4t+3$.
In this paper, we present a new method to construct a near-homogeneous tournament of order $8t+5$,
given any homogeneous tournament of order $4t+3$.
Our method is different from their ones.
It is also a recursive construction, but only uses a copy of $T$ and a copy of $T^*$, the
\textbf{complementary} of $T$, which is obtained by reversing the direction of every arc in $T$,
while the method of Moukouelle needs four copies of $T$.
Next, we extend the definition of near-homogeneous tournament to those with an even number of vertices.
In particular, we find that current definition of near-homogeneous tournament is applicable for tournaments
of order $4t+2$.
Finally we verify path extendability of all near-homogeneous tournaments.

\section{Construction of near-homogeneous tournaments with $n=4t+1$ vertices}
In \cite{Moukouelle1998}, Moukouelle constructed an infinite class of near-homogeneous tournaments
by connecting four copies
of a homogeneous tournaments to an isolated vertex, and adding arcs among them.
In this section, we figure out an original method to construct near-homogeneous tournaments from
just two homogeneous tournaments of the same order.

We use a homogeneous tournament $T$ with $4t+3$ vertices and its complement $T^*$,
which is the key idea of this method.
We now have $8t+6$ vertices. In order to get a graph with $8t+5\equiv 1$ mod $4$ vertices,
we identify a vertex $x\in T$ with the corresponding $x^*\in T^*.$
Then we further reverse the directions of some arcs that incident to $x$,
and add arcs connecting $T\setminus\{x\}$ and $T^*\setminus \{x^*\}$.
The detailed construction is described below.

Let $T$ be a given homogeneous tournament with $4t+3$ vertices.
Let $T^*$ be a copy of the complement of $T$.
For any vertex $u\in T$, we denote the corresponding vertex of $u$ in $T^*$ be $u^*$.
For each vertex $u\in T$, denote the set of in-neighbors and the set of out-neighbors of $u$ in $T$ as $I(u)$ and $O(u),$ respectively; for each vertex $u^*\in T^*$, denote the set of in-neighbors and
the set of out-neighbors of $u^*$ in $T^*$ as $I(u^*)$ and $O(u^*),$ respectively.
Arbitrarily select a vertex $x$ of $T$, delete $x$ from $T$ and delete $x^*$ from $T^*$.
Then, add a new vertex $z$.

Construct a tournament $H$ with the vertex set $V(H)=O(x)\cup I(x) \cup O(x^*) \cup I(x^*) \cup {z}$,
and the arc set defined by the out-neighborhood of every vertex as below:
\\a. For each $u\in O(x),$ the out-neighborhood of $u$ in $H$ is $O(u)\cup (O(u^*)-\{x^*\})\cup \{z\}$;
\\b. For each $u\in I(x),$ the out-neighborhood of $u$ in $H$ is $(O(u)-\{x\})\cup O(u^*)\cup \{u^*\}$;
\\c. For each $u^*\in O(x^*),$ the out-neighborhood of $u^*$ in $H$ is $I(u)\cup O(u^*)$;
\\d. For each $u^*\in I(x^*),$ the out-neighborhood of $u^*$ in $H$ is $(O(u^*)-\{x^*\})\cup (I(u)-\{x\})\cup \{z\}\cup \{u\}$;
\\e. The out-neighborhood of $z$ in $H$ is $I(x)\cup O(x^*)$.
\begin{thm} \label{thm:8t+5}
The tournament $H$ constructed as above is a near-homogeneous tournament of order $8t+5$.
\end{thm}

\begin{proof}
See Appendix \ref{sec:Proof8t+5}.
\end{proof}

To prove the theorem we need to verify that the constructed tournament $H$
is regular and every arc of $H$ is contained in $2t$ or $2t+1$ $3$-cycles.
We leave the routine verification in the Appendix.

\section{Near-homogeneous tournaments of even order}
For convenience, we call a vertex with out-degree $d^+$ and in-degree $d^-$ a $(d^+,d^-)$-vertex.
A tournament $T$ on $n=2k$ vertices
is called \textbf{almost regular} if each vertex $v\in V$
is a $(k,k-1)$-vertex or a $(k-1,k)$-vertex.
We will define near-homogeneous tournaments of even order among almost regular tournaments.

If we classify the arcs in a tournament by the degrees of its endvertices,
then in a regular tournament, there is only one class of arcs,
whose both endvertices are $(\frac{n-1}{2}, \frac{n-1}{2})$-vertices.
Nonetheless, in an almost regular tournament $T$ with $n=2k$ vertices,
there are as many as four types of arcs, named as below.

Class $A$: the arcs $(u,v)$ with $u$ and $v$ being $(k,k-1).$

Class $B$: the arcs $(u,v)$ with $u$ and $v$ being $(k-1,k).$

Class $C$: the arcs $(u,v)$ with $u$ being $(k,k-1)$ and $v$ being $(k-1,k)$.

Class $D$: the arcs $(u,v)$ with $u$ being $(k-1,k)$ and $v$ being $(k,k-1).$

The concept of $3$-cyclic-index can be viewed as a way to further classify the arcs.
By definition, the $3$-cyclic-index of every arc in a homogeneous tournament remains a constant $t$,
and the $3$-cyclic-index of an arc in a near-homogeneous tournaments on $4t+1$ vertices
can be $t$ or $t+1$, therefore the arcs can be further classified into two categories
in a near-homogeneous tournament.
In the case that the tournament is almost regular,
because the arcs have already been classified into four categories by their endvertices,
it sounds reasonable to ask whether the $3$-cyclic-index can remain constant among each class of arcs at first.
Our first result in the following subsection
will show that this could only happen when $n \equiv 2 \pmod 4$.

For further discussion, we calculate the number of arcs in each class in an almost regular
tournament. There are $k$ $(k,k-1)$-vertices, so there are $k \choose 2$ arcs of Class $A$.
Similarly, there are $k \choose 2$ arcs of Class $B$.
Now, for each $(k,k-1)$-vertex $u$, there are $k$ arcs sending out from $u$.
Hence, in total there are $k^2$ arcs sending out from $(k,k-1)$-vertices,
each of which is either in Class $A$ or in Class $C$.
Therefore, there are $k^2- {k \choose 2} =\frac{k(k+1)}{2}={k+1 \choose 2}$ edges in Class $C$.
And there are ${2k \choose 2}-2{k \choose 2}-{k+1 \choose 2}={k \choose 2}$ edges in Class $D$.


For any arc $uv$,
once the class that $uv$ belongs to and the $3$-cyclic-index $\lambda_{uv}$ are determined,
the number of common in-neighbors of $u$ and $v$,
the number of common out-neighbors of $u$ and $v$,
the number of $(u,v)$-$2$-paths and the number of $(v,u)$-$2$-paths
are all determined, as shown in Table \ref{tbl:neighbors}. 

We firstly prove some basic facts about the $3$-cyclic-index of arcs
in each class $A$, $B$, $C$ and $D$, in the form of lemma.

\begin{lem} \label{lem:ClassC=ClassD}
Any $3$-cycle contains exactly an arc of Class $C$ iff it contains exactly an arc of Class $D$,
and thus the sum of $3$-cyclic-index of the arcs in Class $C$ equals that of the arcs in Class $D$.
\end{lem}

\begin{proof}
Suppose an arc $uv$ of Class $C$ is contained in a $3$-cycle $uvwu$.
Since $u$ is a $(k,k-1)$-vertex and $v$ is a $(k-1,k)$-vertex, then either $vw$ or $wu$ must be an arc
of Class $D$. Therefore a $3$-cycle containing an arc of Class $C$ must contain an arc of Class $D$, and vice versa.
Note also that a $3$-cycle can contain at most one arc of Class $C$ and Class $D$, respectively.
Thus we have the conclusion.
\end{proof}

\begin{lem} \label{lem:ABlessthant}
Let $T$ be an almost regular tournament with $n=2k$ vertices.
If the arcs in Class $A$ of $T$ have constant $3$-cyclic-index, denoted by $\lambda_A$,
then we have $\lambda_A \leq \lceil\frac{k}{2}\rceil$.
Similarly, if the arcs in Class $B$ have constant $3$-cyclic-index,
denoted by $\lambda_B$, then we have $\lambda_B \leq  \lceil\frac{k}{2}\rceil$.
\end{lem}

\begin{proof}
We only prove the conclusion for arcs in Class $A$, and for arcs in Class $B$
the conclusion follows similarly.

Let the sub-tournament induced by all the $(k,k-1)$-vertices be $T_1$, and the one induced by all the $(k-1,k)$-vertices be $T_2$. A $2$-path with endvertices as $(k,k-1)$-vertices
either lays entirely in $T_1$, or goes from $T_1$ to $T_2$ and then goes back to $T_1$.
From Table \ref{tbl:neighbors}, the number of 2-paths with endvertices as $(k,k-1)$-vertices and the number of 2-paths with endvertices as $(k-1,k)$-vertices
must be ${k\choose 2}(2\lambda_A-1)$ and ${k\choose 2}(2\lambda_B-1)$, respectively.

First we consider the case when $k=2t+1$.
Since there are $k$ vertices in $T_1$,
and for every $u\in V(T_1)$ the number of $2$-paths go through it is $d^+_{T_1}(u)d^-_{T_1}(u)\le (\frac{k-1}{2})^2$,
the total number of $2$-paths of the first kind is at most $\frac{k(k-1)^2}{4}$.
For a vertex $w$ in $T_2$, the number of $2$-paths of the second kind
which go through it is at most $\frac{k+1}{2} \cdot \frac{k-1}{2}$, hence the total number of $2$-paths of the second kind is at most $\frac{k(k^2-1)}{4}$.
Therefore
$${k \choose 2}(2\lambda_A-1) \le \frac{k(k-1)^2}{4} + \frac{k(k^2-1)}{4} = \frac{k^2(k-1)}{2} = t(2t+1)^2,$$
from which we get $\lambda_A \le t+1=\lceil\frac{k}{2}\rceil$.

Now suppose that $k=2t$.
The number of $2$-paths going through any vertex $u \in V(T_1)$
is $d^+_{T_1}(u)d^-_{T_1}(u)\le \frac{k}{2}\cdot(\frac{k}{2}-1)=\frac{k(k-2)}{4}$,
so the total number of $2$-paths of the first kind is at most $\frac{k^2(k-2)}{4}$.
For a vertex $w$ in $T_2$, the number of $2$-paths of the second kind
which go through it is at most $(\frac{k}{2})^2$,
and the total number of which is at most $k(\frac{k}{2})^2=\frac{k^3}{4}$.
Therefore
$${k \choose 2} (2\lambda_A-1) \le \frac{k^2(k-2)}{4}+\frac{k^3}{4}=t^2(2t-2)+2t^3=2t^2(2t-1),$$
from which we get $\lambda_A \le \frac{2t+1}{2}$. Since $\lambda_A$ is an integer,
$\lambda_A \le t=\lceil\frac{k}{2}\rceil$.
\end{proof}

\begin{table}
\begin{center}
\caption{Classify of vertices according to adjacent relationship
with $u$ and $v$, with their number counted by $3$-cyclic-index
of the arc $uv$ in an almost regular tournament} \label{tbl:neighbors}
\begin{tabular}{ |c|c|c|c|c| }
 \hline
   \makecell{Class of \\ arc $uv$}& \makecell{number \\of common \\out-neighbors \\of $u,v$}
   & \makecell{number \\ of common \\in-neighbors \\of $u,v$}
   & \makecell{number of \\vertices $w$ \\ satisfying \\$u\to w \to v$}
   & \makecell{number of \\vertices $w$\\satisfying \\$v\to w \to u$\\($3$-cyclic-index)}\\
 \hline Class $A$& $k-\lambda_{uv}$&$k-\lambda_{uv}-1$&$\lambda_{uv}-1$&$\lambda_{uv}$\\
\hline Class $B$&$k-\lambda_{uv}-1$&$k-\lambda_{uv}$&$\lambda_{uv}-1$&$\lambda_{uv}$\\
\hline Class $C$&$k-\lambda_{uv}-1$&$k-\lambda_{uv}-1$&$\lambda_{uv}$&$\lambda_{uv}$\\
\hline Class $D$&$k-\lambda_{uv}$&$k-\lambda_{uv}$&$\lambda_{uv}-2$&$\lambda_{uv}$\\
 \hline
\end{tabular}
\end{center}
\end{table}

\subsection{Near-Homogeneous Tournament of Order $\boldsymbol{4t+2}$}

\begin{thm} \label{thm:4t+2}
If the $3$-cyclic-index of each class of arcs remains a constant in an almost regular
tournament $T$ with $n=2k$ vertices, then $n=4t+2$ for certain integer $t\ge 1$.
Furthermore, $\lambda_C = t$
and $\lambda_A = \lambda_B = \lambda_D = t+1$,
where $\lambda_A$, $\lambda_B$, $\lambda_C$, and $\lambda_D$ denote
the $3$-cyclic-index of arcs of class $A,B,C,D$, respectively.
\end{thm}
\begin{proof}
Let $T$, $\lambda_A$, $\lambda_B$, $\lambda_C$ and $\lambda_D$ be as stated in the theorem.
By Lemma \ref{lem:ClassC=ClassD},
the sum of $3$-cyclic-index of arcs in Class $C$ equals that of arc in Class $D$, that is:
\begin{equation}\label{eqn:lambdaClambdaD}
\begin{split}
{k+1 \choose 2}\lambda_C&={k \choose 2}\lambda_D, \\
 \text{and hence }\frac{\lambda_C}{\lambda_D}&=\frac{k-1}{k+1}.
\end{split}
\end{equation}
If $k$ is even then $\gcd(k-1,k+1)=1$, therefore $\lambda_D\ge k+1$.
However, $\lambda_D$ is the number of the intersection of the in-neighborhood of a certain vertex
and the out-neighborhood of another vertex,
and cannot exceed the maximal semi-degree, thus $\lambda_D\le k$, a contradiction.
Therefore $k$ must be odd. Assume that $k=2t+1$ then $n=2k=4t+2$.

Now we have
$$\frac{\lambda_C}{\lambda_D}=\frac{k-1}{k+1}=\frac{t}{t+1}.$$
Since $\lambda_D \le k=2t+1$, the only integer solution to this equality that
is reasonable is
$\lambda_C = t$ and $\lambda_D=t+1$.

Now let's count the number of $2$-paths in $T$.

Every vertex $u \in V(T)$ satisfies that $\{d^+(u), d^-(u)\} = \{k, k-1\}$, so
the number of $2$-paths with $u$ as the intermediate vertex is $k(k-1)$.
Since $|V(T)|=n=2k$, there are totally $2k^2(k-1)$ $2$-paths in $T$.

On the other hand, the number of all kinds of $2$-paths are listed
in the last two columns of Table \ref{tbl:neighbors}.
Thus we can calculate the number of $2$-paths in $T$ by summing up
the last two column for all arcs $uv$, so

\begin{equation} \label{eqn:numberoftwopaths}
\begin{split}
2k^2(k-1)&= {k \choose 2} (2\lambda_A-1) + {k \choose 2} (2\lambda_B-1) +
 {k+1 \choose 2} (2\lambda_C) +  {k \choose 2} (2\lambda_D-2). \\
\end{split}
\end{equation}

Substituting $\lambda_C=t$ and $\lambda_D=t+1$ into (\ref{eqn:numberoftwopaths}) we get
\begin{equation} \label{eqn:lambdaAlambdaB}
\lambda_A+\lambda_B=2t+2.
\end{equation}

By Lemma \ref{lem:ABlessthant}, we have $\lambda_A \le t+1$ and $\lambda_B \le t+1$.
And by (\ref{eqn:lambdaAlambdaB}), equalities must hold, so we have
$$\lambda_A=\lambda_B=t+1.$$
This completes the proof of the theorem.
\end{proof}

It is easy to verify that the tournaments constructed below are as described in Theorem \ref{thm:4t+2}.
\begin{exmp} \label{exmp:4t+2}
Suppose $T$ is a homogeneous tournament with $4t+3$ vertices.
Arbitrarily remove a vertex $x$ from $T$, and obtain a new tournament $T'$.
Then $T'$ is a tournament of order $4t+2$ and each arc of $T$ is has $3$-cyclic-index $t$ or $t+1$.
\end{exmp}

Based on the above result and example, we extend the definition of near-homogeneous tournaments to those with
$4t+2$ vertices.
\begin{defn} \label{defn:4t+2}
A tournament with order $4t+2$, $t\ge 1$ is near-homogeneous if it is almost regular,
and every arc of it has $3$-cyclic-index $t$ or $t+1$.
\end{defn}

Next we handle the case that $n \equiv 0 \pmod 4$.

\subsection{Tournament of Order $\boldsymbol{4t}$}
$\newline$

We first show that the definition of near-homogeneity can not be extended to tournaments
with $n=4t$ ($t\ge 1$) vertices.
\begin{thm}
There does not exist a tournament on $4t$ vertices in which
the $3$-cyclic-index of every arc is $t$ or $t+1$.
\end{thm}
\begin{proof}
Let $T$ be an almost regular tournament with $n=2k=4t$ vertices,
and assume that every arc of $T$ is contained in $t$ or $t+1$ 3-cycles.
Further assume that $m_{C1}$ arcs in Class $C$ have $3$-cyclic-index $t$,
$m_{C2}$ arcs in Class $C$ have $3$-cyclic-index $t+1$,
$m_{D1}$ arcs in Class $D$ have $3$-cyclic-index $t$,
and $m_{D2}$ arcs in Class $D$ have $3$-cyclic-index $t+1$, with $m_{C1}, m_{C2},m_{D1},m_{D2} \geq 0$.

By Lemma \ref{lem:ClassC=ClassD}, we have
\begin{equation} \label{eqn:CequaltoD}
m_{C1} t + m_{C2} (t+1) = m_{D1} t + m_{D2} (t+1),
\end{equation}
in which
\begin{equation} \label{eqn:C1+C2,D1+D2}
m_{C1}+m_{C2}={k+1 \choose 2}, \text{ and }m_{D1}+m_{D2}={k \choose 2}.
\end{equation}

Substituting (\ref{eqn:C1+C2,D1+D2}) into (\ref{eqn:CequaltoD}) we get
$$m_{D2} - m_{C2} = {k+1 \choose 2}t - {k \choose 2}t = 2t^2 > {k \choose 2}.$$

However $m_{D2} \leq {k \choose 2}$, and thus we have $m_{C2} <0$, a contradiction.
\end{proof}

Therefore current definition of near-homogeneous tournament cannot be extended to tournaments with $4t$ vertices.

However, if we replace $t$ and $t+1$ by other values in equation (\ref{eqn:CequaltoD}), the above contradiction
 may be avoided and there can be reasonable solution for (\ref{eqn:CequaltoD}) and (\ref{eqn:C1+C2,D1+D2}).
 Thus, we do not eliminate the possibility that the $3$-cyclic-index of the
 arcs takes two constants throughout the tournament. Therefore we have the following problem.
\begin{prob} \label{prob:TwoValues}
Does there exist a tournament with $4t\ge 16$ vertices in which the $3$-cyclic-index of all arcs takes only two values?
\end{prob}
The reason that we impose the restriction $4t\ge 16$ is that we have used a computer program to check all
almost-regular tournament with $8$ or $12$ vertices, accessed from the Combinatorial Data of Professor Brendan Mckay (\cite{McKay}),
and did not find any of them in which the $3$-cyclic-index of the arcs takes
only two values.
We also have the following properties of the $3$-cyclic-index in tournament of order $n=4t$,
which may be helpful to seek for the answer of Problem \ref{prob:TwoValues}.
\begin{thm} \label{thm:CcannotConstant}
Let $T$ be an almost regular tournament with $n=2k=4t$ vertices.

\noindent (i) The arcs in Class $C$ and the arcs in Class $D$ cannot both have constant $3$-cyclic-index;

\noindent (ii) If the arcs in Class $A$ and Class $B$
have constant $3$-cyclic-index respectively, then the arcs in Class $C$ cannot have
constant $3$-cyclic-index.
\end{thm}
\begin{proof}
If the $3$-cyclic-index remains constants among arcs in Class $C$ and among arcs in Class $D$ respectively,
then from (\ref{eqn:lambdaClambdaD}) and the discussion following it we lead to a contradiction.
Therefore (i) holds.

Now we prove (ii).
Assume contrarily the $3$-cyclic-index of arcs in Class $A$, Class $B$ and Class $C$ remain a constant, respectively.
We use the former notes $\lambda_A$, $\lambda_B$ and $\lambda_C$ to denote these constants. Further we let
the set $\Lambda_D=\{\lambda_{1}, \lambda_{2}, \dots, \lambda_{r}\}$ be all possible $3$-cyclic-index
of arcs in Class $D$, where $m_i$ arcs in Class $D$ is of $3$-cyclic-index $\lambda_{i}$ for $1\le i \le r$,
and $\sum_{i=1}^r m_{i}={k \choose 2}$.

Now equation (\ref{eqn:numberoftwopaths}) becomes
\begin{equation} \label{eqn:numberoftwopathsManyD}
\begin{split}
2k^2(k-1)&= {k \choose 2} (2\lambda_A-1) + {k \choose 2} (2\lambda_B-1) +
 {k+1 \choose 2} (2\lambda_C) +  \sum_{i=1}^{r} m_i (2\lambda_i-2).
\end{split}
\end{equation}

On the other hand, from Lemma \ref{lem:ClassC=ClassD}, we have
\begin{equation}\label{eqn:lambdaClambdaDManyD}
\begin{split}
{k+1 \choose 2}\lambda_C&= \sum_{i=1}^{r} m_i \lambda_i.
\end{split}
\end{equation}

Substituting (\ref{eqn:lambdaClambdaDManyD}) into (\ref{eqn:numberoftwopathsManyD}),
we have
\begin{equation}  \nonumber
\begin{split}
2k^2(k-1)&= {k \choose 2} (2\lambda_A-1) + {k \choose 2} (2\lambda_B-1) +
 {k+1 \choose 2} (2\lambda_C) +  \sum_{i=1}^{r} m_i (2\lambda_i-2). \\
 &= {k \choose 2} (2\lambda_A-1) + {k \choose 2} (2\lambda_B-1) +
 2{k+1 \choose 2} \lambda_C + 2\sum_{i=1}^{r} m_i \lambda_i -2\sum_{i=1}^{r} m_i \\
 &= {k \choose 2} (2\lambda_A-1) + {k \choose 2} (2\lambda_B-1) +
 2{k+1 \choose 2} \lambda_C + 2{k+1 \choose 2} \lambda_C - 2{k \choose 2},
\end{split}
\end{equation}
from which we get
$$2(2t+1)\lambda_C = (2t-1)(4t-\lambda_A-\lambda_B+2)$$
However $\gcd(2(2t+1),2t-1)=1$, thus $\lambda_C$ must be a multiple of $2t-1$.
As $\lambda_C$ cannot exceed the maximal semi-degree, we have $\lambda_C \le k=2t$,
and so $\lambda_C = 2t-1$. Without lose of generality we assume $\lambda_1 = \max \Lambda_D$.
Then we have
$${k+1 \choose 2}\lambda_C= \sum_{i=1}^{r} m_i \lambda_i \le \sum_{i=1}^{r} m_i \lambda_1={k \choose 2} \lambda_1,$$
and that $\lambda_1 \ge k+1$, contradicting that $\lambda_1$ cannot exceed the maximal semi-degree $k$.
The theorem is then proved by this contradiction.
\end{proof}

\section{Path extendability of near-homogeneous tournaments}
In this section, we verify the path extendability of near-homogeneous tournaments.

Let $P=u_0 u_1 \cdots u_{p-1}$ be a non-extendable path of a tournament $T$, and let $v \in V(T) \backslash V(P)$.
There can not exist $0 \leq i \textless j \leq p-1$ such that $u_i \rightarrow v \rightarrow u_j$, or
there must exists $i \le l \le j-1$ such that $u_l \rightarrow v \rightarrow u_{l+1}$ and then
$P'=u_0 \cdots u_lvu_{l+1}u_{p-1}$ extends $P$.
Therefore, according to the direction of the arcs between $V(P)$ and $v$, we can classify $v$ into three categories as follow.

(1) $v \rightarrow V(P)$,

(2) $V(P) \rightarrow v$, and

(3) there exists an integer $s$ such that $v \rightarrow u_i$ for all $0 \leq i \leq s-1$, and $u_i \rightarrow v$ for all $s \leq i \leq p-1$.

We call a vertex $v$ a \textbf{dominating}, \textbf{dominated} or \textbf{hybrid vertex} of $P$ if $v$ belongs to category (1), (2) or (3), respectively.
Furthermore, for a hybrid vertex $v$ in category (3), we say that it \textbf{switches at $s$}.

For a path $P$ and $x,y \in V(P)$, we use $P[x,y]$ to denote the segment of $P$ from $x$ to $y$. Also note that
the notation of \textbf{dominating}, \textbf{dominated} and \textbf{hybrid vertex} is applicable for $\boldsymbol{P[x,y]}$
for a non-extendable path $P=u_0 u_1 \cdots u_{p-1}$.
Firstly, we have the following lemma which limits the number of hybrid vertices of $P[u_1,u_{p-2}]$.

\begin{lem} \label{Lem:NumberHybridV}
Let $P=u_0 u_1 \cdots u_{p-1}$ be a non-extendable path of length $p \geq 4$
in a near-homogeneous tournament $T$ with $n=4t+1$ or $n=4t+2$, and
$P^{\prime} = P[u_1,u_{p-2}]$, then one of the following holds.

(i) $P'$ has no hybrid vertex;

(ii) $P$, and hence $P'$ has at most $(n+1)-(d^+(u_0)+d^-(u_{p-1})) \le i(T) +2$ hybrid vertices.
\end{lem}
\begin{proof}
Let $F = T - V(P)$ and $q = \lvert F \rvert$. Suppose (i) does not hold, then $P'$ has a hybrid vertex $v$,
let's prove (ii).

Note that $v$ is also a hybrid vertex of $P$.
Let the number of hybrid vertices of $P$ be $h \geq 1$.
For every hybrid vertex $w$ of $P$, $u_{p-1} \rightarrow w \rightarrow u_0$,
while the other vertices in $V(T)\backslash V(P)$ dominate $V(P)$ or are dominated by $V(P)$,
so we have $N^{+}_{F}(u_0) \subset N^{+}_{F}(u_{p-1})$ and $h = d_{F}^{+}(u_{p-1}) - d_{F}^{+}(u_0)$.

Since $p \geq 4$, $u_1 \neq u_{p-2}$, and $u_{p-2} \rightarrow v \rightarrow u_1$.
For $1 \leq i \leq p-3$, at most one of $u_0 \rightarrow u_{i+1}$ and $u_i \rightarrow u_{p-1}$ holds, otherwise there exists a path $u_0 u_{i+1} \cdots u_{p-2} v u_1 \cdots u_i u_{p-1}$ which extends $P$, a contradiction. Hence,
$$d_{P}^{+}(u_0) + d_{P}^{-}(u_{p-1}) \leq (p-3) + 1 + 1 +2 = p+1.$$
Since $T$ is a tournament,
\begin{equation} \label{eqn:d+u0d-u(p-1)inF}
\begin{split}
d_{F}^{+}(u_0) + d_{F}^{-}(u_{p-1})
&= d^+(u_0)+d^-(u_{p-1}) - d_{P}^{+}(u_0) - d_{P}^{-}(u_{p-1}) \\
&\geq  d^+(u_0)+d^-(u_{p-1}) - (p+1),
\end{split}
\end{equation}
and
\begin{equation} \label{eqn:degreeu(p-1)inF}
d_{F}^{+}(u_{p-1}) + d_{F}^{-}(u_{p-1}) = q.
\end{equation}
By (\ref{eqn:d+u0d-u(p-1)inF}) and (\ref{eqn:degreeu(p-1)inF}),
\begin{equation}
\begin{split}
h &= d_{F}^{+}(u_{p-1}) - d_{F}^{+}(u_0) \\
  &= (d_{F}^{+}(u_{p-1}) + d_{F}^{-}(u_{p-1}))-(d_{F}^{+}(u_0)+d_{F}^{-}(u_{p-1})) \\
  &\le q - (d^+(u_0)+d^-(u_{p-1}) - (p+1)) \\
  &=(n+1)-(d^+(u_0)+d^-(u_{p-1})),
\end{split}
\end{equation}
which is the first statement in (ii).

To prove the final inequality in (ii) we discuss the following two cases.

Case 1. $n=4t+1$.

Now $T$ is regular, so $d^+(u_0)=d^-(u_{p-1})=(n-1)/2$ and $i(T) = 0$,

So we have
$$(n+1)-(d^+(u_0)+d^-(u_{p-1}))=2 =i(T)+ 2.$$

Case 2. $n=2k=4t+2$.

Now $T$ is almost regular, and each vertex of $T$ is either a $(k,k-1)$-vertex or a $(k-1,k)$-vertex, and $i(T) = 1$.
So $d^+(u_0)\ge k-1$ and $d^-(u_{p-1}) \ge k-1$.
Therefore
$$(n+1)-(d^+(u_0)+d^-(u_{p-1}))\le (n+1)-2(k-1)=3 =i(T)+ 2.$$

This completes the proof of the lemma.
\end{proof}

To prove the path extendability of near-homogeneous tournaments we also need the following result: a digraph is \textbf{arc-3-cyclic} (respectively \textbf{arc-3-anticyclic}), if for every arc $uv$ there is a $(v,u)$-2-path (respectively $(u,v)$-2-path).
We say that a digraph with $n$ vertices is \textbf{completely strong path-connected}, if for every vertex pair $\{ u,v \}$, there are $\{ u,v \}$-paths and $\{ v,u \}$-paths of every length from 2 to $n-1$.

The following theorem is shown in \cite{ZhangK1982}.

\begin{thm} \cite{ZhangK1982} \label{ZhangK}
Let $T$ be a tournament which is arc-3-cyclic and arc-3-anticyclic. Then $T$ is completely strong path-connected, unless $T$ belongs to one class of counterexamples that can be characterized.
\end{thm}

\begin{thm} \label{Thm:PathExt4t+1}
All near-homogeneous tournaments with $n=4t+1 \geq 9$ are path extendable.
\end{thm}
\begin{proof}
Let $T$ be a near-homogeneous tournament with $n=4t+1 \geq 9$ vertices.
Suppose to the contrary that $T$ has a non-extendable path $P=u_0 u_1 \cdots u_{p-1}$.
For $n=9$, we use a computer program to verify that $T$ is path extendable.
The code of the program is listed in Appendix \ref{sec:Program}.
Therefore we may assume $t\ge 3$ and $n \geq 13$.

By Table \ref{tbl:neighbors},
for every arc $uv$ there are at least $t-2 \ge 1$ $\{u,v\}$-$2$-paths, so every arc in $T$ is extendable.
By Theorem \ref{ZhangK}, there is a $(u,v)$-Hamiltonian path in $T$ for every vertex pair $\{ u,v \}$,
therefore every path of $n-1$ vertices is also extendable.
When $p=2$ and $P = u_0 u_1 u_2$, since $u_0 u_1$ can be extended to a path $u_0 x u_1$ and $x \neq u_2$, $P$ can be extended to the path $u_0 x u_1 u_2$, a contradiction. Similarly, when $p=3$, $P$ is extendable.
So we may assume that $p \geq 4$ and $p\le n-2$.

Let $N^{-}$, $N^{+}$ and $N_h$ be the sets of dominating vertices, dominated vertices and hybrid vertices of $P^{\prime} = P[u_1,u_{p-2}]$, and $p^{\prime} = \lvert P^{\prime} \rvert$. Then, $N^{-} \rightarrow u_0$.

Consider the set $S$ of 2-paths in $T$ between the vertices in $V(P^{\prime})$.
Let $w$ be an intermediate vertex of any 2-path in $S$.
Then, $w \in V(P^{\prime}) \cup \{u_0,u_{p-1}\} \cup N_h$.
By Lemma \ref{Lem:NumberHybridV}, $\lvert N_h \rvert \leq i(T) + 2 = 2$.
For $w \in \{u_0,u_{p-1}\} \cup N_h$, the number of 2-paths in $S$ with intermediate vertex $w$ is at most $(p^{\prime} / 2)^2$. For $w \in V(P^{\prime})$, the number of 2-paths in $S$ with intermediate vertex $w$ is at most $((p^{\prime} -1) / 2)^2$. Therefore,
\begin{equation} \label{eqn:UpperboundofS}
\lvert S \rvert \leq 4 \cdot (\frac{p^{\prime}}{2})^2 + p^{\prime} \cdot (\frac{p^{\prime} -1}{2})^2 = \frac{p^{\prime}(p^{\prime} +1)^2}{4}.
\end{equation}

By Table \ref{tbl:neighbors},
for any arc $uv$ of $T$, if $uv$ is contained in $t$ $3$-cycles,
we have $t + t-1 = (n-3)/2$ 2-paths between $u$ and $v$ (of both direction),
and if $uv$ is contained in $t+1$ $3$-cycles,
we have $t + t+1 = (n+1)/2$ 2-paths between $u$ and $v$.
Let $n_1$, $n_2$ be the number of arcs in $\langle P^{\prime} \rangle$ of the above two categories, respectively,
with $n_1 + n_2 = p^{\prime}(p^{\prime} -1)/2$. Then,
\begin{equation} \label{eqn:LowerboundofS}
\lvert S \rvert = n_1 \cdot \frac{n-3}{2} + n_2 \cdot \frac{n+1}{2}
\geq \frac{n-3}{2} \cdot \frac{p^{\prime}(p^{\prime}-1)}{2}
\end{equation}

By (\ref{eqn:UpperboundofS}) and (\ref{eqn:LowerboundofS}), we have
$$\frac{n-3}{2} \cdot \frac{p^{\prime}(p^{\prime}-1)}{2} \leq \frac{p^{\prime}(p^{\prime} +1)^2}{4},$$
that is,
\begin{equation} \label{eqn:SolutionofP'}
{p^{\prime}}^2 - (n-5)p^{\prime} + n-2 \geq 0.
\end{equation}

Since $n \geq 13$, solving $p^{\prime}$ from (\ref{eqn:SolutionofP'}), we have
$$p^{\prime} \geq (n-5 + \sqrt{(n-5)^2-4(n-2)})/2 > (n-5+(n-9))/2 = n-7.$$
Therefore, $p^{\prime} \geq n-6$, and $p \geq n-4$.

Thus $n-4 \le p \le n-2$.
Let $s=n-p$. Since $n \geq 13$ and $T$ is regular, every vertex of $V(T) \backslash V(P)$ must be a
hybrid vertex of $P$, or it can not have balanced in-degrees and out-degree.

Suppose that $3 \leq s \leq 4$, and there are at least 3 vertices outside $P$.
By Lemme \ref{Lem:NumberHybridV}, $P^{\prime}$ has at most $2$ hybrid vertices.
So there exists a vertex $v_i \in V(T) \backslash V(P) (0 \leq i \leq s-1)$, such that $v_i$ is a hybrid vertex of $P$ but not a hybrid vertex of $P^{\prime}$.
Without lose of generality, we assume that $\{ u_1,u_2, \cdots u_{p-1} \} \rightarrow v_i$ and $v_i \rightarrow u_0$.
Then we have $d^{-}(v_i) \geq q-1 \geq n-5 > (n-1)/2$, which contradicts to $d^-(v_i)=(n-1)/2$.

So we have $s=2$.
Let $v_0$ and $v_1$ denote the vertices belong to $V(T) \backslash V(P)$.
Suppose that $v_0$ switches at $u_{r_0}$, $v_1$ switches at $u_{r_1}$.
Without lose of generality, we assume that $r_0 \leq r_1$.
But then there is no $2$-path from $v_0$ to $v_1$, a contradiction.

This completes the proof of the theorem.
\end{proof}

\begin{thm} \label{Thm:PathExt4t+2}
All near-homogeneous tournaments with $n=4t+2 \geq 10$ are path extendable.
\end{thm}
\begin{proof}
Let $T$ be a near-homogeneous tournament with $n=2k=4t+2 \geq 10$ vertices.
Suppose to the contrary that $T$ has a non-extendable path $P=u_0 u_1 \cdots u_{p-1}$.
For $n=10$, similar to the proof of the last theorem, we use the program in Appendix \ref{sec:Program}
to verify that $T$ is path extendable.
Therefore we may assume $t \geq 3$ and $n \geq 14$.
By analog arguments in the proof of the last theorem, we may assume that $p \geq 4$ and $p\le n-2$.

Let $P'$, $p'$, $N^{-}$, $N^{+}$ and $N_h$ be defined in the same way as in Theorem 4.3.
Again we consider the set of 2-paths in $T$ between the vertices in $V(P^{\prime})$, denoted by $S$. Let $w$ be an intermediate vertex of any 2-path in $S$.
Then, $w \in V(P^{\prime}) \cup \{u_0,u_{p-1}\} \cup N_h$.
For $w \in \{u_0,u_{p-1}\} \cup N_h$, the number of 2-paths in $S$ with intermediate vertex $w$ is at most $(p^{\prime} / 2)^2$. For $w \in V(P^{\prime})$, the number of 2-paths in $S$ with intermediate vertex $w$ is at most $((p^{\prime} -1) / 2)^2$.

By Lemma \ref{Lem:NumberHybridV}, $|N_h|\le (n+1)-(d^+(u_0)+d^-(u_{p-1}))=2k+1-(d^+(u_0)+d^-(u_{p-1}))$.
We discuss the following cases according to the degrees of $u_0$ and $u_{p-1}$.

Case 1. $u_0$ and $u_{p-1}$ are both $(k,k-1)$-vertices or both $(k-1,k)$-vertices.

Now $\lvert N_h \rvert \le 2k+1-(d^+(u_0)+d^-(u_{p-1})) =2$, therefore
\begin{equation} \label{eqn:Upperboundcase1}
\lvert S \rvert \leq 4 \cdot (\frac{p^{\prime}}{2})^2 + p^{\prime} \cdot (\frac{p^{\prime} -1}{2})^2 = \frac{p^{\prime}(p^{\prime} +1)^2}{4}.
\end{equation}

We have $2t+1=n/2$ $2$-paths between $u$ and $v$ for every arc $uv$ in Class $A$,
$2t+1=n/2$ $2$-paths between $u$ and $v$ for every arc $uv$ in Class $B$,
$2t=(n-2)/2$ $2$-paths between $u$ and $v$ for every arc $uv$ in Class $C$,
and $2t=(n-2)/2$ $2$-paths between $u$ and $v$ for every arc $uv$ in Class $D$.
Let $n_A$, $n_B$, $n_C$ and $n_D$ be the number of arcs in $\langle P^{\prime} \rangle$ of four classes, respectively,
with $n_A + n_B + n_C + n_D = p^{\prime}(p^{\prime} -1)/2$. Then,
\begin{equation} \label{eqn:Lowerboundcase1}
\lvert S \rvert = \frac{n}{2} \cdot n_A + \frac{n}{2} \cdot n_B + \frac{n-2}{2} \cdot n_C + \frac{n-2}{2} \cdot n_D \geq (\frac{n}{2} - 1) \cdot \frac{p^{\prime}(p^{\prime}-1)}{2}
\end{equation}

By (\ref{eqn:Upperboundcase1}) and (\ref{eqn:Lowerboundcase1}), we have
$$\frac{n-2}{2} \cdot \frac{p^{\prime}(p^{\prime}-1)}{2} \leq \frac{p^{\prime}(p^{\prime} +1)^2}{4},$$
that is,
\begin{equation} \label{eqn:Solutioncase1}
{p^{\prime}}^2 - (n-4)p^{\prime} + n-1 \geq 0.
\end{equation}

Since $n \geq 14$, solving $p^{\prime}$ from (\ref{eqn:Solutioncase1}), we have
$$p^{\prime} \geq (n-4 + \sqrt{(n-4)^2 - 4(n-1)})/2 > (n-4 + (n-15/2)/2 = n - 23/4.$$
Since $p^{\prime}$ is an integer, $p^{\prime} \geq n-5$, and $p \geq n-3$.

Similar to the proof of Theorem \ref{Thm:PathExt4t+1}, when $p=n-3$, there exists a vertex
$v_i \in V(T) \backslash V(P) (0 \leq i \leq 2)$ which is a hybrid vertex of $P$
but not a hybrid vertex of $P'$, and $d^{-}(v_i) \geq n-4 > n/2$ or $d^{+}(v_i) \geq n-4 > n/2$,
which contradicts to that $d^+(v_i), d^-(v_i)\in \{n/2, n/2-1\}$.
When $p=n-2$, we can also get a contradiction by using the same method as in the proof of Theorem \ref{Thm:PathExt4t+1}.

Case 2. $u_0$ is a $(k,k-1)$-vertex, $u_{p-1}$ is a $(k-1,k)$-vertex.

Now $\lvert N_h \rvert \leq 2k+1-(d^+(u_0)+d^-(u_{p-1}))=1$, therefore
\begin{equation} \label{eqn:Upperboundcase2}
\lvert S \rvert \leq 3 \cdot (\frac{p^{\prime}}{2})^2 + p^{\prime} \cdot (\frac{p^{\prime} -1}{2})^2 = \frac{p^{\prime}({p^{\prime}}^2+p^{\prime}+1)}{4}.
\end{equation}

Similarly, by (\ref{eqn:Lowerboundcase1}) and (\ref{eqn:Upperboundcase2}), we have
$$\frac{n-2}{2} \cdot \frac{p^{\prime}(p^{\prime}-1)}{2} \leq \frac{p^{\prime}({p^{\prime}}^2+p^{\prime}+1)}{4},$$
that is,
\begin{equation} \label{eqn:Solutioncase2}
{p^{\prime}}^2 - (n-3)p^{\prime} + n-1 \geq 0.
\end{equation}

Since $n \geq 14$, solving $p^{\prime}$ from (\ref{eqn:Solutioncase2}), we have
$$p^{\prime} \geq (n-3 + \sqrt{(n-3)^2 - 4(n-1)})/2 > (n-3 + (n-6))/2 = n - 9/2.$$
Since $p^{\prime}$ is an integer, $p^{\prime} \geq n-4$, and $p \geq n-2$. Thus $p=n-2$.

We can get a contradiction by using the same method as in the proof of Theorem \ref{Thm:PathExt4t+1}.

Case 3. $u_0$ is a $(k-1,k)$-vertex, $u_{p-1}$ is a $(k,k-1)$-vertex.

Now $\lvert N_h \rvert \leq 2k+1-(d^+(u_0)+d^-(u_{p-1}))=3$, therefore
\begin{equation} \label{eqn:Upperboundcase3}
\lvert S \rvert \leq 5 \cdot (\frac{p^{\prime}}{2})^2 + p^{\prime} \cdot (\frac{p^{\prime} -1}{2})^2 = \frac{p^{\prime}({p^{\prime}}^2+3p^{\prime}+1)}{4}.
\end{equation}

Similarly, by (\ref{eqn:Lowerboundcase1}) and (\ref{eqn:Upperboundcase3}), we have
$$\frac{n-2}{2} \cdot \frac{p^{\prime}(p^{\prime}-1)}{2} \leq \frac{p^{\prime}({p^{\prime}}^2+3p^{\prime}+1)}{4},$$
that is,
\begin{equation} \label{eqn:Solutioncase3}
{p^{\prime}}^2 - (n-5)p^{\prime} + n-1 \geq 0.
\end{equation}

Since $n \geq 14$, solving $p^{\prime}$ from (\ref{eqn:Solutioncase3}), we have
$$p^{\prime} \geq (n-5 + \sqrt{(n-5)^2 - 4(n-1)})/2 > (n-5 + (n-9))/2 = n - 7.$$
Therefore, $p^{\prime} \geq n-6$, and $p \geq n-4$.

When $p=n-4$ and $p=n-2$, we get contradictions by using the same methods
as in the proof of Theorem \ref{Thm:PathExt4t+1} again.

Consider the case when $p=n-3$. Note that every vertex of $V(T) \backslash V(P)$ must be a hybrid vertex of $P$
or its in-degree and out-degree can not be balance.
Let $V(T) \backslash V(P)=\{v_1, v_2, v_3\}$,
assume that $v_i$ switches at $u_{r_i}$ for $i\in \{1,2,3\}$,
 and w.l.o.g. assume that $r_1 \leq r_2 \le r_3$.
Then there is no $2$-path from $v_i$ to $v_j$ with an intermediate vertex in $V(P)$ for $1\le i<j \le 3$.
In particular the only $2$-paths in $T$ from $v_1$ to $v_2$ or $v_3$ must be $v_1v_2v_3$ or $v_1v_3v_2$,
hence there is either no $2$-path from $v_1$ to $v_2$ or no $2$-path from $v_1$ to $v_3$.

However, by the last two column of Table \ref{tbl:neighbors} and the definition of near-homogeneous tournaments
there is at least $t-2 \ge 1$ $2$-paths from $u$ to $v$ for any vertex pair $u$ and $v$, a contradiction.

This completes the proof of the theorem.
\end{proof}

\section{Conclusion}
Homogeneous tournaments are highly symmetric tournaments that are path extendable.
Their order must be $4t+3$ for certain integer $t$.
In order to find more path extendable tournaments,
we look for similar symmetry in tournaments of other orders,
by exploring the concept of near-homogeneity in this paper.

While near-homogeneous tournament has been defined for tournaments with order $4t+1$,
we first put forward a new method to construct near-homogeneous tournaments of order $8t+5=4(2t+1)+1$
when a homogeneous tournament of order $4t+3$ is given.
Next, we extend the definition of near-homogeneity to tournaments with even order.
In particular, there exist tournament of order $4t+2$ which satisfy
the original definition of near-homogeneity that each arc is contained in $t$ or $t+1$ $3$-cycles.
However, we prove that no tournament of order $4t$ satisfies the definition.
It remains a problem whether it is possible that
the number of $3$-cycles containing each arc in a tournament of order $4t$ takes only two values.
Finally we verify path extendability of near-homogeneous tournaments of order $4t+1$ and $4t+2$.

While the existence, construction and enumeration of homogeneous tournaments of every
possible order remain challenging problems,
we can find interesting problems concerning near-homogeneous tournaments as well.
As we can see, the existing constructions of near-homogeneous tournaments
are mostly transformations from given homogeneous tournaments
(e.g. the one given in \cite{Moukouelle1998}, Theorem \ref{thm:8t+5} and Example \ref{exmp:4t+2}),
so we are interested in the constructions
of near-homogeneous tournaments that are not based on homogeneous tournaments 
(e.g. the one given in \cite{AstieDugat1992}).
In particular, for near-homogeneous tournaments with an even order, we ask for the
existence of tournaments in this class, besides the ones that we construct in Example
\ref{exmp:4t+2}.

\section*{Acknowledgments}
We would like to thank Qi Zhang for his assistance in preparing our verification programs.

\appendix






\section{Proof of theorem \ref{thm:8t+5}} \label{sec:Proof8t+5}

First, we prove that $H$ is regular. $H$ contains $2(4t+3)-1=8t+5$ vertices.
We verify that the out-neighborhood of every vertex $u$ in $H$ is of order $4t+2$,
thus $d^+(u)=d^-(u)=4t+2$.

For any $u\in O(x),$ the set of out-neighbors of $u$ is $O(u)\cup (O(u^*)-\{x^*\})\cup \{z\}$,
which has order $(2t+1)+2t+1=4t+2$.
For any $u\in I(x),$ the set of out-neighbors of $u$ is $(O(u)-\{x\})\cup O(u^*)\cup \{u^*\}$,
which has order $2t+(2t+1)+1=4t+2$.
For any $u^*\in O(x^*),$ the set of out-neighbors of $u^*$ is $I(u)\cup O(u^*),$
which has order $2(2t+1)=4t+2$.
For any $u^*\in I(x^*),$ the set of out-neighbors of $u^*$ is $O(u^*-\{x^*)\cup (I(u)-x)\cup \{z\}\cup \{u\},$ which has order $2(2t)+1+1=4t+2.$
Finally, the set of out-neighbors of $z$ is $I(x)\cup O(x^*),$ which has order $2(2t+1)=4t+2.$

Next, to prove that $H$ is near-homogeneous, we only need to verify that every pair of vertices
$u$ and $v$ have $2t$ or $2t+1$ common out-neighbors.
For, without lose of generality we may assume that $u \rightarrow v$.
Since $H$ is a tournament and $v$ is of out-degree $4t+2$, the number of $3$-cycles that contains the arc $uv$
is then $4t+2-2t=2t+2$ or $4t+2-(2t+1)=2t+1$.
We will divide our discussion into 14 cases according to the vertex subset that $u$ and $v$ belongs to.
Note that $T^*$ is also homogeneous, and in a homogenous tournament of order $4t+3$,
every two distinct vertices have $t$ common in-neighbors and $t$ common out-neighbors.
Furthermore for any arc $uv$ in $T$, there are $t$ $(u,v)$-paths of length $2$ in $T$, i.e.
$|O(u)\cap I(v)|=t$.
To improve the readability of our proof, we will write it in a relatively formal form and if the vertex
$u$ or $v$ is from $T^*$, we write it as $u^*$ and $v^*$.

Case 1. $v=z$ and $u\in O(x)$.

The out-neighborhood of $v$ in $H$ is $I(x)\cup O(x^*)$,
and the out-neighborhood of $u$ in $H$ is $O(u)\cup (O(u^*)-\{x^*\}) \cup \{z\}$ where $u^*\in I(x^*)$.

$T$ is homogeneous $\Longrightarrow$ $xu$ is contained in $t+1$ $3$-cycles in $T$
$\Longrightarrow$ $|O(u)\cap I(x)|=t+1$.

$T^*$ is homogeneous $\Longrightarrow$ $u^*$ and $x^*$ have $t$ common out-neighbors
$\Longrightarrow$ $|(O(u^*)-\{x^*\})\cap O(x^*)|=t$.

$|O(u)\cap I(x)|=t+1$ and $|(O(u^*)-\{x^*\})\cap O(x^*)|=t$
$\Longrightarrow$ $u$ and $v$ have $t+1+t=2t+1$ common out-neighbors in $H$.

Case 2. $v=z$ and $u\in I(x)$.

The out-neighborhood of $v$ in $H$ is $I(x)\cup O(x^*)$,
and the out-neighborhood of $u$ in $H$ is $(O(u)-\{x\})\cup O(u^*)\cup \{u^*\} $ where $u^*\in O(x^*)$.

$T$ is homogeneous $\Longrightarrow$ $\langle I(x) \rangle$ is regular
$\Longrightarrow$ $|(O(u)-\{x\})\cap I(x)|=t$.

$T^*$ is homogeneous $\Longrightarrow$ $\langle O(x^*) \rangle$ is regular
$\Longrightarrow$ $|O(u^*)\cap O(x^*)|=t$.

$u^*\in O(x^*)$ $\Longrightarrow$ $v \rightarrow u^*$.

$|O(u)\cap I(x)|=t$, $|(O(u^*)-\{x\})\cap O(x^*)|=t$, $u\rightarrow u^*$ and $v\rightarrow u^*$
$\Longrightarrow$ $u$ and $v$ have $2t+1$ common out-neighbors in $H$.

Case 3. $v=z$ and $u^*\in O(x^*)$.

The out-neighborhood of $v$ in $H$ is $I(x)\cup O(x^*)$,
and the out-neighborhood of $u^*$ in $H$ is $I(u)\cup O(u^*)$ where $u\in I(x)$.

$T$ is homogeneous $\Longrightarrow$ $\langle I(x) \rangle$ is regular
$\Longrightarrow$ $|I(u)\cap I(x)|=t$.

$T^*$ is homogeneous $\Longrightarrow$ $\langle O(x^*) \rangle$ is regular
$\Longrightarrow$ $|O(u^*)\cap O(x^*)|=t$.

$|I(u)\cap I(x)|=t$ and $|O(u^*)\cap O(x^*)|=t$ $\Longrightarrow$
$u$ and $v$ have $2t$ common out-neighbors in $H$.

Case 4. $v=z$ and $u^*\in I(x^*):$

The out-neighborhood of $v$ in $H$ is $I(x)\cup O(x^*)$,
and the out-neighborhood of $u^*$ in $H$ is $(O(u^*)-\{x^*\})\cup (I(u)-x)\cup \{z\}\cup \{u\}$ where $u\in O(x)$.

$T$ is homogeneous $\Longrightarrow$ $u$ and $x$ have $t$ common in-neighbors in $T$
$\Longrightarrow$ $|I(x)\cap (I(u)-x)|=t$.

$T^*$ is homogeneous $\Longrightarrow$ $u^*$ and $x^*$ have $t$ common out-neighbors in $T^*$
$\Longrightarrow$ $|(O(u^*)-\{x^*\}) \cap O(x^*)|=t$.

$|I(x)\cap I(u)|=t$ and $|O(u^*) \cap O(x^*)|=t$ $\Longrightarrow$
$u$ and $v$ have $2t$ common out-neighbors in $H$.

Case 5. $u,v\in O(x)$.

The out-neighborhood of $u$ in $H$ is $O(u)\cup (O(u^*)-\{x^*\})\cup \{z\}$ where $u^*\in I(x^*)$,
and the out-neighborhood of $v$ in $H$ is $O(v)\cup (O(v^*)-\{x^*\})\cup \{z\}$ where $v^*\in I(x^*)$.

$T$ is homogeneous $\Longrightarrow$ $u$ and $v$ have $t$ common out-neighbors in $T$
$\Longrightarrow$ $|O(u)\cap O(v)| = t$.

$T^*$ is homogeneous $\Longrightarrow$ $u^*$ and $v^*$ have $t$ common out-neighbors in $T^*$ including
$x^*$ $\Longrightarrow$ $|(O(u^*)-\{x^*\})\cap (O(v^*)-\{x^*\})| = t-1$.

$|O(u)\cap O(v)| = t$, $|(O(u^*)-\{x^*\})\cap (O(v^*)-\{x^*\})| = t-1$ and $z$ is a common out-neighbor of $u$ and $v$
$\Longrightarrow$ $u$ and $v$ have $2t$ common out-neighbors in $H$.

Case 6. $u,v\in I(x)$.

The out-neighborhood of $u$ in $H$ is $(O(u)-\{x\})\cup O(u^*)\cup \{u^*\}$ where $u^*\in O(x^*)$,
and the out-neighborhood of $v$ in $H$ is $(O(v)-\{x\})\cup O(v^*)\cup \{v^*\}$ where $v^*\in O(x^*)$.

$T$ is homogeneous $\Longrightarrow$ $u$ and $v$ have $t$ common out-neighbors in $T$ including $x$
$\Longrightarrow$ $|(O(u)-\{x\})\cap (O(v)-\{x\})| = t-1$.

$T^*$ is homogeneous $\Longrightarrow$ $u^*$ and $v^*$ have $t$ common out-neighbors in $T^*$
$\Longrightarrow$ $|O(u^*) \cup O(v^*)|=t$.

Without lose of generality, we assume that $u^*\rightarrow v^*$ in $T^*$ $\Longrightarrow$
$v^*$ is a common out-neighbor of $u$ and $v$.

$|(O(u)-\{x\})\cap (O(v)-\{x\})| = t-1$, $|O(u^*) \cup O(v^*)|=t$
and $v^*$ is a common out-neighbor of $u$ and $v$ $\Longrightarrow$
$u$ and $v$ have $(t-1)+t+1=2t$ common out-neighbors in $H$.

Case 7. $u\in O(x),v\in I(x)$.

The out-neighborhood of $u$ in $H$ is $O(u)\cup (O(u^*)-\{x^*\})\cup \{z\}$ where $u^*\in I(x^*)$,
and the out-neighborhood of $v$ in $H$ is $(O(v)-\{x\})\cup O(v^*)\cup \{v^*\}$ where $v^*\in O(x^*)$.

$T$ is homogeneous $\Longrightarrow$ $u$ and $v$ have $t$ common out-neighbors in $T$ not including $x$
$\Longrightarrow$ $|O(u)\cap (O(v)-\{x\})|=t$.

$T^*$ is homogeneous $\Longrightarrow$ $u^*$ and $v^*$ have $t$ common out-neighbors in $T^*$ not including $x^*$
$\Longrightarrow$ $|(O(u^*)-\{x^*\})\cap O(v^*)|=t$.

If $u^*\rightarrow v^*$ then $v^*$ is a common out-neighbor of $u$ and $v$, else $v^*$ is not a common out-neighbor
of $u$ and $v$.

$|O(u)\cap (O(v)-\{x\})|=t$, $|(O(u^*)-\{x^*\})\cap O(v^*)|=t$ and $v^*$ may or may not be a common out-neighbor
of $u$ and $v$ $\Longrightarrow$ $u$ and $w$ have $2t$ or $2t+1$ common out-neighbors in $H$.

Case 8. $u^*,v^*\in O(x^*)$.

The out-neighborhood of $u^*$ is $I(u)\cup O(u^*)$ where $u\in I(x)$,
and the out-neighborhood of $v^*$ is $I(v)\cup O(v^*)$ where $v \in I(x)$.

$T$ is homogeneous $\Longrightarrow$ $u$ and $v$ have $t$ common in-neighbors in $T$ not including $x$
$\Longrightarrow$ $|I(u)\cap I(v)|=t$.

$T^*$ is homogeneous $\Longrightarrow$ $u^*$ and $v^*$ have $t$ common out-neighbors in $T$ not including $x^*$
$\Longrightarrow$ $|O(u^*)\cap O(v^*)|=t$.

$|I(u)\cap I(v)|=t$ and $|O(u^*)\cap O(v^*)|=t$
$\Longrightarrow$ $u^*$ and $v^*$ have $2t$ common out-neighbors in $H$.

Case 9. $u^*,w^*\in I(x^*)$.

The out-neighborhood of $u^*$ is $(O(u^*)-\{x^*\})\cup (I(u)-\{x\})\cup \{z\}\cup \{u\}$ where $u\in O(x)$,
and the out-neighborhood of $v^*$ is $(O(v^*)-\{x^*\})\cup (I(v)-\{x\})\cup \{z\}\cup \{v\}$ where $v\in O(x)$.

$T$ is homogeneous $\Longrightarrow$ $u$ and $v$ have $t-1$ common in-neighbors besides $x$
$\Longrightarrow$ $|(I(u)-\{x\})\cap (I(v)-\{x\})|=t-1$.

$T^*$ is homogeneous $\Longrightarrow$ $u^*$ and $v^*$ have $t-1$ common out-neighbors besides $x^*$
$\Longrightarrow$ $|(O(u^*)-\{x^*\})\cap (O(v^*)-\{x^*\})|=t-1$.

Without lose of generality we may assume that $u\rightarrow v$, i.e. $u\in I(v)$ $\Longrightarrow$
$v^* \rightarrow u$ $\Longrightarrow$ $u$ is a common out-neighbor of $u^*$ and $v^*$.

$|(I(u)-\{x\})\cap (I(v)-\{x\})|=t-1$, $|(O(u^*)-\{x^*\})\cap (O(v^*)-\{x^*\})|=t-1$,
$u$ and $z$ are two common out-neighbors of $u^*$ and $v^*$ $\Longrightarrow$
$u^*$ and $v^*$ have $(t-1)+(t-1)+2=2t$ common out-neighbors in $H$.

Case 10. $u^*\in O(x^*)$, $v^*\in I(x^*)$.

The out-neighborhood of $u^*$ in $H$ is $I(u)\cup O(u^*)$ where $u\in I(x)$, and
the out-neighborhood of $v^*$ in $H$ is $(O(v^*)-\{x^*\})\cup (I(v)-\{x\})\cup \{z\}\cup \{v\}$
where $v \in O(x)$.

$T$ is homogeneous $\Longrightarrow$ $u$ and $v$ have $t$ common in-neighbors in $T$ not including $x$
$\Longrightarrow$ $|I(u)\cap (I(v)-\{x\})|=t$.

$T^*$ is homogeneous $\Longrightarrow$ $u^*$ and $v^*$ have $t$ common out-neighbors in $T^*$
not including $x^*$ $\Longrightarrow$ $|O(u^*) \cap (O(v^*)-\{x^*\})|=t$.

If $v\rightarrow u$ in $T$, i.e. $v\in I(u)$ then $u^*\rightarrow v$ and $v$ is a common out-neighbor
of $u^*$ and $v^*$, else $v$ is not a common out-neighbor of $u^*$ and $v^*$.

$|I(u)\cap (I(v)-\{x\})|=t$, $|O(u^*) \cap (O(v^*)-\{x^*\})|=t$
and $v$ may or may not be a common out-neighbor of $u^*$ and $v^*$
$\Longrightarrow$ $u^*$ and $v^*$ have $2t$ or $2t+1$ common out-neighbors in $H$.

Case 11. $u\in O(x),$ $v^*\in O(x^*)$.

The out-neighborhood of $u$ is $O(u)\cup (O(u^*)-\{x^*\})\cup \{z\}$ where $u^*\in I(x^*)$,
and the out-neighborhood of $v^*$ is $I(v)\cup O(v^*)$ where $v \in I(x)$.

$T$ is homogeneous $\Longrightarrow$
If $u\rightarrow v$ then $|O(u)\cap I(v)|=t$, and if $v\rightarrow u$ then $|O(u)\cap I(v)|=t+1$,
where $x\notin O(u)\cap I(v)$.

$T^*$ is homogeneous $\Longrightarrow$ $u^*$ and $v^*$ have $t$ common out-neighbors
not including $x$ $\Longrightarrow$ $|(O(u^*)-\{x^*\})\cap O(v^*)|=t$.

$|O(u)\cap I(v)|=t$ or $t+1$ and $|(O(u^*)-\{x^*\})\cap O(v^*)|=t$
$\Longrightarrow$ $u$ and $v^*$ have $2t$ or $2t+1$ common out-neighbors in $H$.

Case 12. $u\in I(x),$ $v^*\in I(x^*)$.

The out-neighborhood of $u$ is $(O(u)-\{x\})\cup O(u^*)\cup \{u^*\}$ where $u^*\in O(x^*)$, and
 the out-neighborhood of $v^*$ is $(O(v^*)-\{x^*\})\cup (I(v)-\{x\})\cup \{z\}\cup \{v\}$ where $v\in O(x)$.

Suppose that $u\rightarrow v$, then $v^* \rightarrow u^*$.
$T$ is homogeneous $\Longrightarrow$ there are $t$ paths of length $2$ from $u$ to $v$
including $uxv$ $\Longrightarrow$ $|(O(u)-\{x\}) \cap ((I(v)-\{x\}))|=t-1$.
Furthermore $v$ and $u^*$ are common out-neighbors of $u$ and $v^*$.

Suppose that $v\rightarrow u$, then $u^*\rightarrow v^*$.
$T$ is homogeneous $\Longrightarrow$ there are $t+1$ paths of length $2$ from $u$ to $v$
including $uxv$ $\Longrightarrow$ $|(O(u)-\{x\}) \cap ((I(v)-\{x\}))|=t$.
Furthermore $v$ and $u^*$ are not common out-neighbors of $u$ and $v^*$.

$T^*$ is homogeneous $\Longrightarrow$ $u^*$ and $v^*$ have $t$ common out-neighbors not including $x$
$\Longrightarrow$ $|O(u^*)\cap (O(v^*)-\{x^*\})| = t$.

Therefore, $u^*$ and $v^*$ have $(t-1)+2+t=2t+1$ or $t+t=2t$ common out-neighbors in $H$.

Case 13. $u\in O(x),$ $v^*\in I(x^*)$.

The out-neighborhood of $u$ is $O(u)\cup (O(u^*)-\{x^*\})\cup \{z\}$ where $u^*\in I(x^*)$,
and the out-neighborhood $v^*$ is $(O(v^*)-\{x^*\})\cup (I(v)-\{x\})\cup \{z\}\cup \{v\}$ where $v\in O(x)$.

If $v^*=u^*,$ then the common out-neighbors of $v^*$ and $u$ are precisely
$z$ and the $2t$ vertices in $O(v^*)-\{x^*\}.$
Thus, $u$ and $v^*$ have $2t+1$ common out-neighbors in $H$.

Now assume that $v^*\neq u^*$.

Suppose that $u\rightarrow v$, then $v$ is a common out-neighbor of $u$ and $v^*$.
$T$ is homogeneous $\Longrightarrow$ there are $t$ paths of length $2$ from $u$ to $v$,
not including $uxv$
$\Longrightarrow$ $|O(u)\cap ((I(v)-\{x\})\cap \{v\})|=t+1$.

Suppose that $v\rightarrow u$, then $v$ is not a common out-neighbor of $u$ and $v^*$.
$T$ is homogeneous $\Longrightarrow$
there are $t+1$ paths of length $2$ from $u$ to $v$, not including $uxv$
$\Longrightarrow$ $|O(u)\cap (I(v)-\{x\})|=t+1$.

$T^*$ is homogeneous $\Longrightarrow$ $u^*$ and $v^*$ have $t$ common out-neighbors in $T^*$ including $x^*$
$\Longrightarrow$ $|(O(u^*)-\{x^*\})\cap (O(v^*)-\{x^*\})|=t-1$.

Since $z$ is also a common out-neighbor of $u$ and $v^*$, in either case we conclude that
$u$ and $v^*$ have $(t+1)+(t-1)+1=2t+1$ common out-neighbors in $H$.

Case 14. $u\in I(x),$ $v^*\in O(x^*)$.

The out-neighborhood of $u$ is $(O(u)-\{x\})\cup O(u^*)\cup \{u^*\}$ where $u^*\in O(x^*)$,
and the out-neighborhood of $v^*$ is $I(v)\cup O(v^*)$ where $v\in I(x)$.

If $v^*=u^*$, then the common out-neighbors of $v^*$ and $u$ are precisely the $2t+1$ vertices in $O(u^*)$,
thus $u$ and $v^*$ have $2t+1$ common out-neighbors in $H$.

Now assume that $v^*\neq u^*$.

If $u \rightarrow v$, then $v^*\rightarrow u^*$ and $u^*$ is a common out-neighbor of $u$ and $v^*$.

$T$ is homogeneous $\Longrightarrow$
there are $t$ paths of length $2$ from $u$ to $v$, not including $uxv$
$\Longrightarrow$ $|(O(u)-\{x\}) \cap I(v)|=t$.

$T^*$ is homogeneous $\Longrightarrow$
$u^*$ and $v^*$ have $t$ common out-neighbors in $T^*$, not including $x^*$
$\Longrightarrow$ $|O(u^*)\cap O(v^*)|=t$.

Totally, $u$ and $v^*$ have $1+t+t=2t+1$ common out-neighbors in $H$.

If $v \rightarrow u$, then $u^*\rightarrow v^*$ and $u^*$ is not a common out-neighbor of $u$ and $v^*$.

$T$ is homogeneous $\Longrightarrow$
there are $t+1$ paths of length $2$ from $u$ to $v$, not including $uxv$
$\Longrightarrow$ $|(O(u)-\{x\}) \cap I(v)|=t+1$.

$T^*$ is homogeneous $\Longrightarrow$
$u^*$ and $v^*$ have $t$ common out-neighbors in $T^*$, not including $x^*$
$\Longrightarrow$ $|O(u^*)\cap O(v^*)|=t$.

Totally, $u$ and $v^*$ have $(t+1)+t=2t+1$ common out-neighbors in $H$.

Summarizing all cases, every pair of vertices in $H$ have $2t$ or $2t+1$ common out-neighbors.
Therefore, $H$ is near-homogeneous and the theorem is proved.

\section{Verification Program for near-homogeneous tournaments of oreder $9$ and $10$} \label{sec:Program}
In this section we give the code to verify Theorem \ref{Thm:PathExt4t+1} and Theomre \ref{Thm:PathExt4t+2} for $t=2$,
that is, a near-homogeneous tournament with $n=9$ or $n=10$ vertices is path extendable.

There are two programs used. One is in Matlab and the other is in C++.
We take the data of regular and almost regular tournaments on $9$ or $10$ vertices
from the graph database of Professor Brendan Mckay (\cite{McKay}).
These graph data are then used as the input of the Matlab program,
which works as a filter that outputs all graphs satisfying the definition of near-homogeneous tournaments.
The C++ program, accepting the output from the Matlab program, works as a verifier
which check that every graph output from the filter is path extendable.

\begin{lstlisting}[language=Matlab,caption={Matlab: Filtering data that matches the definition}]
clear all;
clc;
t=2; %n=4t+1 or n=4t+2
nrow = 10; % When the number of vertices is 9, nrow=9

% Read file
file = fopen("C:\Users\DELL\Desktop\data.txt");
data = textscan(file, '%s', 'Delimiter', '\n');
fclose(file);
lines = data{1};

result=[];
for gid = 1 : length(lines)
    % Generate adjacency matrices based on the data for each row
    line = lines{gid};
    mat = zeros(nrow);
    for k = 1 : length(line)
        r = ceil(((2 * nrow - 1) - sqrt((2 * nrow - 1)^2 - 8 * k)) / 2);
        c =k - (2 * nrow - r) * (r - 1) / 2 + r;
        mat(r, c) = str2double(line(k));
        mat(c, r) = 1 - mat(r, c);
    end

    % Determine which graphs fit the definition
    breakflag = 0;
    for i = 1:nrow
        for j = 1:nrow
            if(mat(i,j) == 1)
                tmpvec = transpose(mat(:,i)) + mat(j,:);
                if(sum(tmpvec==2) < t || sum(tmpvec==2) > t+1)
                    breakflag = 1;
                    break;
                end
            end
        end
        if breakflag == 1
            break;
        end
    end

    if breakflag == 0
       result = [result; line]; % Record data that matches the definition
    end
end
\end{lstlisting}

\begin{lstlisting}[language=C++,caption={C++: Verifying path extendability}]
#include <fstream>
#include <iostream>
#include <vector>
#include <array>
#include <algorithm>

using P2 = std::array<int, 2>;
using VEC2 = std::vector<P2>;
using P3 = std::array<int, 3>;
using VEC3 = std::vector<P3>;
using P4 = std::array<int, 4>;
using VEC4 = std::vector<P4>;
using P5 = std::array<int, 5>;
using VEC5 = std::vector<P5>;
using P6 = std::array<int, 6>;
using VEC6 = std::vector<P6>;
using P7 = std::array<int, 7>;
using VEC7 = std::vector<P7>;
using P8 = std::array<int, 8>;
using VEC8 = std::vector<P8>;
using P9 = std::array<int, 9>;
using VEC9 = std::vector<P9>;
using P10 = std::array<int, 10>;
using VEC10 = std::vector<P10>;

#define SHOWGRAPH
for(int i=0; i<10; ++i)
{
    for(int j=0; j<10; ++j)
    {
        std::cout << mat[i][j] << " ";
    }
    std::cout << std::endl;
}

#define SHOWPATH(path)
for(auto& p : path)
{
    std::cout << p << " ";
}
std::cout << std::endl;

#define SHOWVEC(VEC)
for(auto& vec : VEC)
{
    for(auto& val : vec)
    {
        std::cout << val << " ";
    }
    std::cout << std::endl;
}

bool belong(const P10& arr, const int len, int val)
{
    for(int i=0; i<len; ++i)
    {
        if(arr[i] == val)
        {
            return true;
        }
    }

    return false;
}

template<typename TYP1, typename TYP2>
bool AllElementsIn(const TYP1& arr1, const TYP2& arr2)
{
    //Use std::all_of to check that every element in arr1 exists in arr2
    return std::all_of(arr1.begin(), arr1.end(), [&arr2](int elem){return std::find(arr2.begin(), arr2.end(), elem) != arr2.end();});
}

void generate_path(const int    start,
                   const int    end,
                   const int    mat[10][10],
                   VEC2&        vec2,
                   VEC3&        vec3,
                   VEC4&        vec4,
                   VEC5&        vec5,
                   VEC6&        vec6,
                   VEC7&        vec7,
                   VEC8&        vec8,
                   VEC9&        vec9,
                   VEC10&       vec10)
{
    vec2.clear();
    vec3.clear();
    vec4.clear();
    vec5.clear();
    vec6.clear();
    vec7.clear();
    vec8.clear();
    vec9.clear();
    vec10.clear();

    VEC10   vectmp0, vectmp1;
    vectmp0.clear();
    vectmp1.clear();
    vectmp0.push_back({start});

    //Generate VEC2
    for(auto& vec : vectmp0)
    {
        int i = vec[0];
        for(int j=0; j<10; ++j)
        {
            if(mat[i][j])
            {
                if(!belong(vec, 1, j))
                {
                    if(j == end)
                    {
                        vec2.push_back({i, j});
                    }
                    else
                    {
                        vectmp1.push_back({i, j});
                    }
                }
            }
        }
    }
    std::swap(vectmp0, vectmp1);
    vectmp1.clear();
    //SHOWVEC(vec2);

    //Generate VEC3
    for(auto& vec : vectmp0)
    {
        int i = vec[1];
        for(int j=0; j<10; ++j)
        {
            if(mat[i][j])
            {
                if(!belong(vec, 2, j))
                {
                    if(j == end)
                    {
                        vec3.push_back({vec[0], i, j});
                    }
                    else
                    {
                        vectmp1.push_back({vec[0], i, j});
                    }
                }
            }
        }
    }
    std::swap(vectmp0, vectmp1);
    vectmp1.clear();
    //SHOWVEC(vec3);

    //Generate VEC4
    for(auto& vec : vectmp0)
    {
        int i = vec[2];
        for(int j=0; j<10; ++j)
        {
            if(mat[i][j])
            {
                if(!belong(vec, 3, j))
                {
                    if(j == end)
                    {
                        vec4.push_back({vec[0], vec[1], i, j});
                    }
                    else
                    {
                        vectmp1.push_back({vec[0], vec[1], i, j});
                    }
                }
            }
        }
    }
    std::swap(vectmp0, vectmp1);
    vectmp1.clear();
    //SHOWVEC(vec4);

    //Generate VEC5
    for(auto& vec : vectmp0)
    {
        int i = vec[3];
        for(int j=0; j<10; ++j)
        {
            if(mat[i][j])
            {
                if(!belong(vec, 4, j))
                {
                    if(j == end)
                    {
                        vec5.push_back({vec[0], vec[1], vec[2], i, j});
                    }
                    else
                    {
                        vectmp1.push_back({vec[0], vec[1], vec[2], i, j});
                    }
                }
            }
        }
    }
    std::swap(vectmp0, vectmp1);
    vectmp1.clear();
    //SHOWVEC(vec5);

    //Generate VEC6
    for(auto& vec : vectmp0)
    {
        int i = vec[4];
        for(int j=0; j<10; ++j)
        {
            if(mat[i][j])
            {
                if(!belong(vec, 5, j))
                {
                    if(j == end)
                    {
                        vec6.push_back({vec[0], vec[1], vec[2], vec[3], i, j});
                    }
                    else
                    {
                        vectmp1.push_back({vec[0], vec[1], vec[2], vec[3], i, j});
                    }
                }
            }
        }
    }
    std::swap(vectmp0, vectmp1);
    vectmp1.clear();
    //SHOWVEC(vec6);

    //Generate VEC7
    for(auto& vec : vectmp0)
    {
        int i = vec[5];
        for(int j=0; j<10; ++j)
        {
            if(mat[i][j])
            {
                if(!belong(vec, 6, j))
                {
                    if(j == end)
                    {
                        vec7.push_back({vec[0], vec[1], vec[2], vec[3], vec[4], i, j});
                    }
                    else
                    {
                        vectmp1.push_back({vec[0], vec[1], vec[2], vec[3], vec[4], i, j});
                    }
                }
            }
        }
    }
    std::swap(vectmp0, vectmp1);
    vectmp1.clear();
    //SHOWVEC(vec7);

    //Generate VEC8
    for(auto& vec : vectmp0)
    {
        int i = vec[6];
        for(int j=0; j<10; ++j)
        {
            if(mat[i][j])
            {
                if(!belong(vec, 7, j))
                {
                    if(j == end)
                    {
                        vec8.push_back({vec[0], vec[1], vec[2], vec[3], vec[4], vec[5], i, j});
                    }
                    else
                    {
                        vectmp1.push_back({vec[0], vec[1], vec[2], vec[3], vec[4], vec[5], i, j});
                    }
                }
            }
        }
    }
    std::swap(vectmp0, vectmp1);
    vectmp1.clear();
    //SHOWVEC(vec8);

    //Generate VEC9
    for(auto& vec : vectmp0)
    {
        int i = vec[7];
        for(int j=0; j<10; ++j)
        {
            if(mat[i][j])
            {
                if(!belong(vec, 8, j))
                {
                    if(j == end)
                    {
                        vec9.push_back({vec[0], vec[1], vec[2], vec[3], vec[4], vec[5], vec[6], i, j});
                    }
                    else
                    {
                        vectmp1.push_back({vec[0], vec[1], vec[2], vec[3], vec[4], vec[5], vec[6], i, j});
                    }
                }
            }
        }
    }
    std::swap(vectmp0, vectmp1);
    vectmp1.clear();
    //SHOWVEC(vec9);

    //Generate VEC10
    for(auto& vec : vectmp0)
    {
        int i = vec[8];
        for(int j=0; j<10; ++j)
        {
            if(mat[i][j])
            {
                if(!belong(vec, 9, j))
                {
                    if(j == end)
                    {
                        vec10.push_back({vec[0], vec[1], vec[2], vec[3], vec[4], vec[5], vec[6], vec[7], i, j});
                    }
                    else
                    {
                        vectmp1.push_back({vec[0], vec[1], vec[2], vec[3], vec[4], vec[5], vec[6], vec[7], i, j});
                    }
                }
            }
        }
    }
    std::swap(vectmp0, vectmp1);
    vectmp1.clear();
    //SHOWVEC(vec10);

}

bool check_path(const int gid,
                const VEC2&    vec2,
                const VEC3&    vec3,
                const VEC4&    vec4,
                const VEC5&    vec5,
                const VEC6&    vec6,
                const VEC7&    vec7,
                const VEC8&    vec8,
                const VEC9&    vec9,
                const VEC10&   vec10)
{
    for(auto& path2 : vec2)
    {
        for(auto& path3 : vec3)
        {
            if(AllElementsIn<P2, P3>(path2, path3))
            {
                for(auto& path4 : vec4)
                {
                    if(AllElementsIn<P3, P4>(path3, path4))
                    {
                        for(auto& path5 : vec5)
                        {
                            if(AllElementsIn<P4, P5>(path4, path5))
                            {
                                for(auto& path6 : vec6)
                                {
                                    if(AllElementsIn<P5, P6>(path5, path6))
                                    {
                                        for(auto& path7 : vec7)
                                        {
                                            if(AllElementsIn<P6, P7>(path6, path7))
                                            {
                                                for(auto& path8 : vec8)
                                                {
                                                    if(AllElementsIn<P7, P8>(path7, path8))
                                                    {
                                                        for(auto& path9 : vec9)
                                                        {
                                                            if(AllElementsIn<P8, P9>(path8, path9))
                                                            {
                                                                for(auto& path10 : vec10)
                                                                {
                                                                    if(AllElementsIn<P9, P10>(path9, path10))
                                                                    {
                                                                        //SHOWPATH(path2);
                                                                        //SHOWPATH(path3);
                                                                        //SHOWPATH(path4);
                                                                        //SHOWPATH(path5);
                                                                        //SHOWPATH(path6);
                                                                        //SHOWPATH(path7);
                                                                        //SHOWPATH(path8);
                                                                        //SHOWPATH(path9);
                                                                        //SHOWPATH(path10);
                                                                        return true;
                                                                    }
                                                                }
                                                            }
                                                        }
                                                    }
                                                }
                                            }
                                        }
                                    }
                                }
                            }
                        }
                    }
                }
            }
        }
    }

    return false;
}

int main()
{
    //Variable assignment
    std::fstream     fstr;
    std::string      line;
    std::string      lines[379];
    //There are 379 rows of data that fit the definition by Matlab
    int              id = 0;
    bool             flag = true;
    VEC2             vec2;
    VEC3             vec3;
    VEC4             vec4;
    VEC5             vec5;
    VEC6             vec6;
    VEC7             vec7;
    VEC8             vec8;
    VEC9             vec9;
    VEC10            vec10;

    //Read file
    fstr.open("/home/CUDA/data.txt", std::ios_base::in);
    while(std::getline(fstr, line))
    {
        lines[id] = line;
        ++id;
    }
    fstr.close();

    for(int i=0; i<379; ++i)
    {
        line = lines[i];

        //Generating matrix
        int mat[10][10] = {0};
        for(int k=0; k<line.length(); ++k)
        {
             int r = std::floor((19 - std::sqrt(361 - 8 * k)) / 2);
             int c =k - (20 - r - 1) * r / 2 + r + 1;
             mat[r][c] = line[k] - '0';
             mat[c][r] = 1 - mat[r][c];
        }
        //SHOWGRAPH

        //Given s, t begins to proceed
        for(int start=0; start<10; ++start)
        {
            for(int end=0; end<10; ++end)
            {
                if(mat[start][end]) //There is an edge between two vertices
                {
                    //Generation path
                    generate_path(start, end, mat, vec2, vec3, vec4, vec5, vec6, vec7, vec8, vec9, vec10);
                    //SHOWVEC(vec3);

                    //Check path
                    flag = check_path(i, vec2, vec3, vec4, vec5, vec6, vec7, vec8, vec9, vec10);
                    if(flag)
                    {
                        std::cout << "graph " << i << ", start= " << start << ", end = " << end << " is ok!\n";
                    }
                    else
                    {
                        std::cout << "graph " << i << ", start= " << start << ", end = " << end << " not ok: \n";
                        SHOWGRAPH;
                        break;
                    }
                }
            }

            if(!flag)
            {
                break;
            }
        }
    }

    return 0;
}
\end{lstlisting}

\end{document}